\documentclass[11pt]{article}
\usepackage{geometry}  
\usepackage{color, graphicx}
\usepackage{epstopdf}
\usepackage[all]{xy}

\usepackage{amsmath}

\usepackage{charter}
\usepackage[T1]{fontenc}
\usepackage{textcomp}
\usepackage{calrsfs}

\newcommand{\runninghead}[1]{\gdef\RH{#1}}\newcommand{\RH}{}
\newcommand{\msc}[1]{\gdef\MSC{#1}}\newcommand{\MSC}{}
\renewcommand{\title}[1]{\gdef\TT{#1}}\newcommand{\TT}{}
\renewcommand{\author}[1]{\gdef\AU{#1}}\newcommand{\AU}{}
\newcommand{\address}[1]{\gdef\ADR{#1}}\newcommand{\ADR}{}
\renewcommand{\date}[1]{\gdef\DD{#1}}\newcommand{\DD}{}
\newcommand{\version}[1]{\gdef\VER{#1}}\newcommand{\VER}{}

\renewcommand{\maketitle}{
\par \noindent {\footnotesize \upshape Running head: \RH \hfill \DD \par
\noindent Math.\ Subj.\ Class. (2000): \MSC \hfill\textbf{\VER}} 
\vspace{3cm}
\begin{center} {\LARGE \normalfont \TT} \par \medskip \AU \end{center} \vspace{1cm}
}

\newcommand{\finalinfo}{
\bigskip \noindent {\small \upshape \ADR} }

\usepackage{sectsty}

\allsectionsfont{\normalsize \bfseries \nohang \centering}
\makeatletter
\def\@seccntformat#1{\csname the#1\endcsname. \ }
\makeatother

\numberwithin{equation}{section}

\usepackage{amsthm}
\theoremstyle{plain}
\newtheorem{introprop}{Proposition}
\newtheorem{introthm}[introprop]{Theorem}

\swapnumbers
\theoremstyle{plain}
\newtheorem{thm}{Theorem}[section]
\newtheorem{prop}[thm]{Proposition}

\newtheorem{lem}[thm]{Lemma}
\theoremstyle{definition}
\newtheorem{df}[thm]{Definition}

\newtheorem{rem}[thm]{Remark}
\newtheorem{se}[thm]{}

\usepackage{amssymb}
\usepackage{bbm}

\def\Z{\mathbf{Z}}

\def\R{\mathbf{R}}
\def\C{\mathbf{C}}

\def\P{\mathbf{P}}

\def\til#1{\widetilde{#1}}% tilde
\def\Aut{\mathop{\mathrm{Aut}}\nolimits}

\def\tr{\mathop{\mathrm{tr}}\nolimits}

\def\PGL{\mathop{\mathrm{PGL}}\nolimits}

\def\PSL{\mathop{\mathrm{PSL}}\nolimits}

\def\cS{{\mathcal S}}

\usepackage[british]{babel}

\begin{document}

\runninghead{Spectral triples for Riemann surfaces}
\version{version 1.2}
\date{\today} 
\msc{ 20H10,  57S30, 58B34}
\title{Zeta functions that hear the shape \\[1mm]  of a Riemann surface}
\author{\textit{by} Gunther Cornelissen \textit{and} Matilde Marcolli}
 
\address{(gc) Mathematisch Instituut, Universiteit Utrecht, Postbus
80.010, 3508 TA Utrecht, Nederland,\\ email: {cornelis@math.uu.nl}  

\medskip

\noindent (mm) Max-Planck-Institut f\"ur Mathematik, Vivatsgasse 7,
53111 Bonn, Deutschland, \\ email: {marcolli@mpim-bonn.mpg.de}} 
\maketitle

\begin{abstract} 
\noindent To a compact hyperbolic Riemann surface, we associate a
finitely summable spectral triple whose underlying topological space
is the limit set of a corresponding Schottky group,  and whose
``Riemannian'' aspect (Hilbert space and Dirac operator) encode the
boundary action through its Patterson-Sullivan measure. We prove that
the ergodic rigidity theorem for this boundary action implies that the
zeta functions of the spectral triple suffice to characterize the
(anti-)conformal isomorphism class of the corresponding Riemann
surface. Thus, you can hear the shape of a Riemann surface, by
listening to a suitable spectral triple. 
\end{abstract}

\section*{Introduction} % || OFF ||
% \section{Introduction % || ON ||

Let $X$ denote a compact Riemann surface of genus $g \geq 2$. By the
retrosection theorem of Koebe-Courant (e.g., \cite{Bers}) $X$ can be
represented by a Schottky group $\Gamma$: we can write $$X=\Gamma
\backslash(\P^1(\C)-\Lambda_\Gamma),$$ where $\Lambda=\Lambda_\Gamma$
is the set of limit points of $\Gamma$ and $\Gamma$ is a free group of
rank $g$, discrete in $\PGL(2,\C)$. In complex analysis, it is well
known that the dynamics of the action of $\Gamma$ on the limit set
endowed with Patterson-Sullivan measure encodes a lot about the
structure of the Riemann surface. Our purpose is to show that this
action can be conveniently encoded by a notion from non-commutative
geometry, namely a \emph{spectral triple} (\cite{ConnesLMP}) which
provides the non-commutative analogue of a Riemannian manifold. As it
will turn out, the spectral triple we will consider is
commutative. But, as has been observed frequently, ``even for
classical spaces, which correspond to commutative algebras, the new
point of view [of noncommutative geometry] will give new tools and
results.'' (Connes, \cite{Connesbook}, p.\ 1).  
We show that the isometry class of the boundary action is encoded in
the \emph{zeta function} formalism of the spectral triple.  

\paragraph{Construction of the spectral triple} At the
topological level, we consider the commutative algebra $A=C(\Lambda)$.
We also need the dense subalgebra $A_\infty=C(\Lambda,\Z)\otimes_\Z
\C$ of locally constant functions on $\Lambda$. 

It might be natural to consider the boundary operator algebra
$C(\Lambda) \rtimes \Gamma$ instead. However, by the non-amenability
of the group $\Gamma$, the hyperfiniteness result of Connes implies
that this algebra does not carry any finitely summable spectral triple
(\cite{Conneshyper}). As we will indicate at the end of the proof of
the main theorem, our construction can be extented to an AF-algebra
that is Morita equivalent to a large subalgebra of the boundary
operator algebra. 

To retrieve the actual conformal structure, we need to make the
operator algebra act on a Hilbert space in a way compatible with a
Dirac operator. The Hilbert space $H$ is a particular
GNS-representation of $A$. Its construction depends on chosing a
state, and we make this choice in such a way that it encodes the
metric action of $\Gamma$ on $\Lambda$, expressed via the
Patterson-Sullivan measure. More specifically, on certain elements of
$A$ related to words in a presentation of $\Gamma$, it gives the
measure of the subset of $\Lambda$ reached from that word in  the
representation of the limit set of $\Gamma$ via word group completion
(Floyd \cite{Floyd}). Finally, the Dirac operator $D$ is composed from
projection operators depending on the word length grading in a
presentation for $\Gamma$. Let $\cS_X$ denote the spectral triple so
constructed (see Section \ref{s1} for details). 

\begin{introprop} \label{tA}
If $X$ is a compact Riemann surface of genus at least $2$, $\cS_X$ is a 1-summable spectral triple.
\end{introprop}

\paragraph{Zeta function rigidity} The theory of finitely summable
spectral triples comes with an elegant framework of zeta functions:
for any $a \in A_\infty$, one has the spectral zeta function
$\zeta_{X,a}(s):=\mathrm{tr} (a|D|^s)$. 

\begin{introthm}\label{tB}
If $X_1$ and $X_2$ are compact Riemann surfaces of genus at least $2$,
such that  
$\zeta_{X_1,a}(s)=\zeta_{X_2,a}(s)$ for all $a \in A_\infty$, then
$X_1$ and $X_2$ are conformally or anti-conformally equivalent as
Riemann surfaces. In particular, the spectral triple $\cS_X$ encodes
the conformal/anticonformal isomorphism type of $X$. \end{introthm}  

In the theorem, equality of zeta functions should be understood as
follows: both the algebras $A_1$ and $A_2$ of $X_1$ and $X_2$,
respectively, have a unit. If the zeta functions for this unit are
equal, we first deduce that the Riemann surfaces have the same
genus. This allows us to identify the algebras $A_1$ and $A_2$, along
with the corresponding dense subalgebras of locally constant
functions. It then makes sense to interpret the expression
$\zeta_{X_1,a}(s)=\zeta_{X_2,a}(s)$. 

About the proof: by the analogue of Fenchel-Nielsen theory (cf.\ Tukia
\cite{Tukia2}), the abstract isomorphism of Schottky groups of $X_1$
and $X_2$ induces a unique homeomorphism of limit sets, equivariant
with respect to the group isomorphism (the \emph{boundary map}). We
show that equality of zeta functions implies that this boundary map is
absolutely continuous. For this, one has to trace through the
representation of the limit set in the sense of Floyd (\cite{Floyd}) to deduce fairly explicit expressions for various zeta functions. We do this by computing traces in an explicit orthonormal basis for $H$.

We then apply an ergodic rigidity theorem for Schottky uniformization
that we deduce from a theorem of Yue (cf.\ \cite{Yue}; from the long
history we also mention the names Mostow, Kuusalo, Bowen, Sullivan,
Tukia).  This says that there are only two alternatives for the
boundary map: either the Patterson-Sullivan measures are mutually
singular with respect to the boundary map, or the map extends to a
continuous automorphism of $\PGL(2,\C)$. Absolute continuity excludes
the first case.  

We will end the paper with a list  of open questions. 

\paragraph{Can you hear the shape of a Riemann surface?} Theorem
\ref{tB} fits into the framework of isospectrality questions, as
coined by I.M.~Gelfand for compact Riemannian manifolds. Vign\'eras (\cite{V1}, \cite{V2})  and Sunada
(\cite{Sunada}) constructed non-isometric surfaces with identical Laplace operator zeta function $\tr(\Delta^s)$ (`isospectral'). The work of Sunada, in particular, transports an idea from algebraic number theory due to Gassmann (\cite{Gassmann}), where the same phenomenon is visible: there exist non-isomorphic algebraic number fields with identical Dedekind zeta function.

 For the specific case of Riemann surfaces, Buser (\cite{Buser}) obtained  a
(finite) upper bound on the number of Riemann surfaces isospectral with a given Riemann surface, depending only on the genus of that surface, and work of Brooks, Gornet and Gustafson (\cite{BGG})
shows this bound is of the correct order of magnitude. 

As for Dirac operators instead of Laplace operators, B\"ar has
constructed non-isometric space forms with the same Dirac spectrum
(\cite{Baer1}).  

For the case of planar domains, the problem of isospectrality was
coined by Bochner, to quote Kac --- quoting Lipman Bers ---  ``can you
hear the shape of a drum?'' (\cite{Kac},  solved by Gordon, Webb and
Wolpert \cite{GWW}). 

In this phrasing, Theorem \ref{tB} says you can
hear the complex analytic type of a compact Riemann surface from
listening to the noncommutative spectra of its associated spectral
triple (that is, to the collection of the associated zeta
functions). A main difference with respect to the classical isospectrality
question is that in this case you do have to listen to $\tr(aD^s)$ for a
dense subset of operators $a \in A$, and the eigenvalues of $D$
themselves are not so interesting. For example, at the unit $1 \in A$, we find the innocent zeta function (cf.\ Formula (\ref{zeta1}))
$$ \zeta_{X,1}(s) = 1 +\frac{2g-2}{2g-1} \cdot \frac{(2g)^{3s+1}}{1-(2g-1)^{3s+1}}. $$
 We note that in the completely
different construction of a  ``conformal'' spectral triple by B\"ar,
the eigenvalues themselves are uninteresting, too (\cite{Baer}). 

Thus, the main point of our discussion is that the ``spectral object'' $\cS_X$ determines the ``conformal object'' $X$, up to complex conjugation.

%\begin{ack} 
%\end{ack}

\section{A spectral triple associated to a Kleinian Schottky group} \label{s1}

\begin{se}The aim of this section is to introduce a finitely summable spectral
triple $\cS_X:=(A,H,D)$ associated to a Schottky group $\Gamma$ that
uniformizes a compact Riemann surface $X$ of genus $g \geq 2$.  Recall
that a spectral triple is a noncommutative analogue of a 
Riemannian spin manifold, 
where $A$ is a $C^*$-algebra, $H$ is a Hilbert space on which
$A$ acts by bounded operators, and $D$ is an unbounded self
adjoint operator on $H$ with compact resolvent $(D-z)^{-1}$ for
$z\notin \R$, and such that the commutators $[D,a]$ are bounded
operators for all $a$ in a dense involutive subalgebra $A_\infty$ of
$A$. Connes has shown \cite{ConnesLMP} that if $(A,H,D)$ arises from a
Riemannian spin manifold, then the distance element is encoded by the
inverse of the Dirac operator.  
\end{se}

\begin{se}Let $\Gamma$ denote a Schottky group of rank $g \geq 2$.  As an abstract group, $\Gamma$ is isomorphic to $F_g$, the free group on $g$ generators. 
We think of $\Gamma$ as being specified by an injective group
homomorphism $\rho  :  F_g \hookrightarrow \PGL(2,\C)$. Let
$\Lambda=\Lambda_\Gamma$ denote the limit set of the action of
$\Gamma$ on $\P^1(\C)$. 
 \end{se}

\begin{se}[Group completion and limit set] We recall what we need from
Floyd's relation between the group completion of $F_g$ and the limit
set $\Lambda$ of $\Gamma$ (\cite{Floyd}).  Let $Y_g$ denote the Cayley
graph (with unordered edges) of $F_g$ for a presentation of $F_g$ in a
fixed alphabet on $g$ letters, and let $\bar{Y}_g$ denote the
completion of the Cayley graph as a metric space for the following
metric. Let $|w|$ denote the reduced word length of a word in the
generators of $F_g$. The edge between two words $w_1$ and $w_2$ is
given length $\min \{ |a|^{-2}, |b|^{-2} \}$ (with $|e|^{-2}:=1$ for
the empty word $e$).  
The \emph{group completion} of $F_g$ is by definition the space
$\bar{F}_g:=\bar{Y}_g-Y_g$. It is a compact metric space. A different
(finite) presentation for $F_g$ leads to a Lipshitz equivalent group
completion. Since $F_g$ has no ``parabolic ends'' in the sense of
Floyd, we have the following: 
\end{se}

\begin{lem}[Floyd, \cite{Floyd}, p.\ 213--217] Given a point $x_0 \in
\P^1(\C)$ and an embedding $\rho  :  F_g \hookrightarrow \PGL(2,\C)$
as above, the following map is a continuous bijection: 
$$ \begin{array}{lccc} \iota_\rho \ : \  & \bar{F}_g & \rightarrow &
\Lambda \\ &  \lim\limits_i w_i & \mapsto & \lim\limits_i
\rho(w_i)(x_0). \end{array} \ \ \ \ \Box $$  
\label{Floyd} \end{lem}

\begin{df} Given a reduced word $w$ in the generators of $F_g$, let
$i(w)$ respectively $t(w)$ denote the initial, respectively terminal
letter of $w$. 
For two reduced words $w$ and $v$ (or $v$ a limit of such), we write $$w \subseteq v \mbox{ if }(\exists w_0)(v=w \cdot w_0) \mbox{ with }t(w) \neq i(w_0)^{-1}.$$
 We write $w \subset v$ if $w \subseteq v$ and $w \neq v$. 

Given $\rho: F_g \rightarrow \PGL(2,\C)$ and a word $w \in F_g$,
define the subset of $\Lambda$ of \emph{ends of $w$ with respect to
$\rho$} to be  
$$ \overrightarrow{w}_\rho := \{ \iota_\rho(v) \ : \ v \in \bar{F}_g \mbox{ and } w \subseteq v \}. $$ 
\end{df}

\begin{lem}
$A=C(\Lambda)$ is the closure of the span of the characteristic
functions $\chi_{\overrightarrow{w}_\rho}$ of the sets
$\overrightarrow{w}_\rho$ for $w \in F_g$.  
\end{lem}

\begin{proof}
This is immediate, since $\Lambda$ is a totally disconnected compact
Hausdorff space, and the sets $\overrightarrow{w}$ form a basis of
clopen sets for its topology. 
\end{proof}

We denote by $A_\infty$ the dense involutive subalgebra of $A$ spanned
by the characteristic functions $\chi_{\overrightarrow{w}_\rho}$. 

\begin{df} Let $\mu_\Lambda$ denote the Patterson-Sullivan measure on
$\Lambda$ (cf. \cite{Patterson}, \cite{SullivanIHES}). Its main property is scaling by the Hausdorff dimension
$\delta_H$ of $\Lambda$: \begin{equation*}\label{PSmeas} 
(\gamma^* d\mu)(x)= |\gamma^\prime (x)|^{\delta_H} \, d\mu(x), \ \ \
\forall \gamma\in \Gamma.
\end{equation*} We define a state $\tau : A_\infty \rightarrow \R$ by 
$$ \tau(\chi_{\overrightarrow{w}_\rho}) := \int\limits_{\Lambda}
\chi_{\overrightarrow{w}_\rho} \mathrm{d}\mu_\Lambda =
\mu_\Lambda(\overrightarrow{w}_\rho). $$ 
The above lemma shows that $\tau$ extends uniquely to a state on
$A$. We define the Hilbert space $H$ to be the GNS-representation of
$A$ arising from this state $\tau$, that 
is, the completion of $A$ with respect to the inner product $\langle
a | b \rangle := \tau(b^*a)$.  
\end{df} 

\begin{df} We now take our inspiration from the construction of
Antonescu and Christensen in \cite{AntChris}.  
The subalgebra $A_\infty$ of $A=C(\Lambda)$ is a limit of finite
dimensional subspaces $A_\infty =
\lim\limits_{\xrightarrow[]{}} A_n$ with $A_n$ the span of the
characteristic functions of sets of ends of reduced words of length
$\leq n$. This filtration is inherited by  $H$. We denote by $H_n$ the
term of the filtration of $H$ corresponding to $A_n$ and we let
$P_n$ denote the orthogonal projection operator onto $H_n$. 
We define the Dirac operator to be 
$$ D:= \lambda_0 P_0 + \sum_{n \geq 1} \lambda_n (P_n-P_{n-1}), $$
where $\lambda_n = (\dim A_n)^3$. Note that $Q_n:=P_n-P_{n-1}$ is the
projection onto the graded pieces, identified with the orthogonal
complements $H_n \ominus H_{n-1}$, which correspond to words of exact 
length $n$. The choice of $\lambda_n$ arises from the fact that we
then arrive at 1-summability (Proposition \ref{tA}): 
\end{df}

\begin{prop}
The triple $\cS_{X}=(A,H,D)$ is a 1-summable spectral
triple.
\end{prop}
\begin{proof} The $*$-operation is complex conjugation, and since $D$
is real, it is self-adjoint. For $a\in A_{n}$ and for any $m>n$, multiplication by $a$ maps $A_{m-1}$ and $A_m$ into itself. Therefore, $a$ commutes with the projections $P_m$ and $P_{m-1}$ and so $[Q_m,a]=0$. Hence 
$$ [D,a] = \sum_{i=0}^n \lambda_i [Q_i,a]$$
is a finite sum (we set $P_{-1}=0$ for convenience). 
Thus, the commutators of $D$ with elements in  
the dense subalgebra $A_{\infty}$ of $A$ are bounded. 

Moreover, one has $\dim A_{n} \geq
n+1$, hence the 1-summability (and compact resolvent): 
$$ \tr((1+D^2)^{-1/2})=1 +\sum_{n=1}^\infty (1+\lambda_n^2)^{-1/2} (\dim
H_{n}-\dim H_{n-1}) $$
$$ \leq 1+\sum_{n=1}^\infty (1+\lambda_n^2)^{-1/2}\dim  H_{n}
\leq 1+\sum_{n=1}^\infty (1+\lambda_n^2)^{-1/2}\dim  A_{n} $$
$$ \leq 1+\sum_{n=1}^\infty (\dim  A_{n})^{-2} \leq  
1+\sum_{n=1}^\infty (n+1)^{-2} \leq  2, $$ where we used $\lambda_n =
(\dim A_n)^3$ in the second-to-last inequality. This proves the
proposition. 
\end{proof}

\begin{rem}
A recent deep theorem of Rennie and V\'arilly (\cite{RV}) allows one
to decide whether a given spectral triple is associated to an actual
commutative Riemannian spin manifold. For the purpose of this paper,
since we are mostly interested in the zeta functions, we do not
consider any additional structure on the spectral triple. In
particular, our Dirac operator is only considered up to sign, since
the sign does not play a role in the zeta functions, while for
\cite{RV}, the sign provides the essential information on the
$K$-homology fundamental class. It is possible that our construction
may be refined to incorporate the further necessary properties of an
abelian spectral triple to which the reconstruction theorem can be
applied. In that case, it seems that the underlying metric geometry
should probably relate to the existence of quasi-circles of limit sets
of Schottky groups as in  \cite{Bowen} --- see also the next
remark. \end{rem} 

\begin{rem}
Notice that our construction provides a 1-summable spectral triple on
the limit set, regardless of the actual value of its Hausdorff
dimension (which can be greater than one). Thus, the metric dimension
seen from this construction will be in general different from the
actual metric dimension of the limit set embedded in $\P^1(\C)$. The
existence in all cases of a 1-summable spectral triple on the limit
set reflects the fact that topologically $\Lambda$ is always a Cantor
set that can be embedded in a topologically 1-dimensional quasi-circle
(Bowen \cite{Bowen}). In the metric induced by the embedding in
$\P^1(\C)$, the quasi-circle need not be rectifiable (when the
Hausdorff dimension of the limit set exceeds 1), but the existence of
1-summable spectral triples is compatible with the topological
dimension being one in all cases.   
\end{rem}

\section{Boundary isometry from the
spectral zeta function} \label{boundary}

In this section we study the effect of equality of zeta functions on
metric properties of the limit sets. Since we are dealing with two
Riemann surfaces $X_1, X_2$, we will now sometimes index symbols
($H,D,\zeta,\lambda,\dots)$ by the index of the corresponding Riemann
surface and will do so without further mention. If there is no index,
we refer to any of the two Riemann surfaces. 

As was already observed by Connes for the spectral triple associated
to a usual spin manifold, only the action of $A$ on $H$, or of $D$ on
$H$, doesn't capture interesting (metrical/conformal)  information
about the space, it is the  interaction of the action of $A$ and $D$
that is important (\cite{Connesbook}, VI.1). For our purposes, this
interaction will be encoded in the framework of zeta functions of
spectral triples. The zeta functions are $\zeta_a(s):=\tr(a|D|^s)$, a
priori defined for $\mathrm{Re}(s)$ sufficiently negative, but then
meromorphically extended to the whole complex plane with poles at the
dimension spectrum of the spectral triple (see \cite{ConnesLMP}). 

\begin{thm}\label{zetalambda}
Let $X_1$ and $X_2$ be compact Riemann surfaces of genus at least
$2$. If $ \zeta_{a,{X_1}}(s) = \zeta_{a,X_2}(s)$ for all $a \in
A_\infty$, then $g_1=g_2$ and $$ \forall \eta \in F_g \ : \
\mu_1(\overrightarrow{\eta}_{\rho_1}) =
\mu_2(\overrightarrow{\eta}_{\rho_2}). $$ \end{thm}

\begin{rem}
As was indicated in the introduction, equality of zeta functions
should be understood as follows: both the algebras $A_1$ and $A_2$ of
$X_1$ and $X_2$, respectively, have a unit. If the zeta functions for
this unit are equal, then we will conclude from this that the Riemann
surfaces have the same genus. Therefore, the algebras $A_1$ and $A_2$
are isomorphic via the homeomorphism $\Phi: \Lambda_1 \rightarrow
\Lambda_2$ induced from the Floyd maps in the triangle 
$$ \xymatrix{    & & \Lambda_1 \ar[dd]^{\Phi} \\ \bar{F}_g
\ar[urr]^{\iota_{\rho_1}} \ar[drr]^{\iota_{\rho_2}} & \\ & &
\Lambda_2}$$ It then makes sense to interpret the expression
$\zeta_{X_1,a}(s)=\zeta_{X_2,a}(s)$ for elements $a\neq 1$ in
$A_\infty$. 
\end{rem}

\begin{proof} We make the convention that all words are reduced.

Let $X$ be a Riemann surface of genus $g \geq 2$ and $\cS_X$ its
associated spectral triple.  Suppose given an element $a=\chi_U$ in
$A_\infty$. We can assume $U=\overrightarrow{\eta}$ for a given word
$\eta$ of length $|\eta|=m$, since any  $a$ is a linear combination of such.  

We now construct an orthogonal basis for $H$.
First, we prove a lemma about ends of words.

\begin{lem} \label{lemwords} Let $w_1, w_2$ denote two words. If
$\overrightarrow{w_1} \cap \overrightarrow{w_2} \neq \emptyset$, then
$\overrightarrow{w}_1 \subseteq \overrightarrow{w}_2$ or conversely
$\overrightarrow{w}_2 \subseteq \overrightarrow{w}_1$. In particular, if we set $\max\{w,v\}$ to be the largest of the words $w$ and $v$ (if they are comparable in the order $\subseteq$) and $\emptyset$ otherwise, we find that $$ \overrightarrow{w_1} \cap \overrightarrow{w_2} = \overrightarrow{\max\{w_1,w_2\}}$$ with the convention 
$\overrightarrow{\emptyset}=\emptyset$, see Figure \ref{figure1}. 
\begin{figure}[h] \begin{center} \input{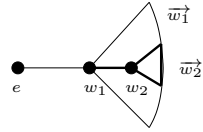}  
\end{center}
\caption{{Illustration of Lemma \ref{lemwords}.}}
\label{figure1}
\end{figure}
 
\end{lem}
\begin{proof}   

If this were not the case, then there is an end that lands in
the nonempty intersection and starts from both segments $w_1$ and
$w_2$, leading to a loop in the Cayley graph $Y_g$, but $Y_g$ is a
tree, so this is impossible, see Figure \ref{figure2} (as usual, we identify the limit set $\Lambda$ topologically with $\bar{F}_g$, by Floyd's Lemma).

\begin{figure}[h] \begin{center} \input{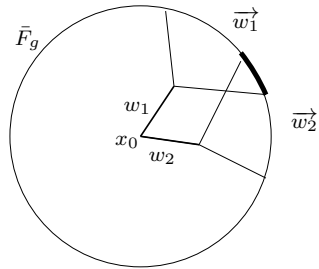} 
\end{center}
\caption{{ A forbidden situation.}}
\label{figure2}
\end{figure}

\end{proof} 

It is then easy to find a basis for the individual spaces $H_n$. 

\begin{lem}\label{basis}
The functions $\chi_w$ for $|w|=n$ for a linear basis for $H_n$, and $$ \langle \chi_w |  \chi_v \rangle = 
\mu(\overrightarrow{\max\{v,w\}}). $$
 \end{lem}

\begin{proof} The characteristic functions $\chi_{\overrightarrow{w}}$
for $|w|=n$ give a linear basis for $H_n$: they are
linearly independent as their supports are disjoint, and they generate the space since for any word $u$ of length $|u|< n$ one
has $$\chi_{\overrightarrow{u}} =\sum_{\substack{|w|=n \\u \subset w}}
\chi_{\overrightarrow{w}}.$$ Indeed, by the previous Lemma, all occuring $\chi_{\overrightarrow{w}}$ have disjoint support, and their union $\bigcup \overrightarrow{w}$ equals $\overrightarrow u$.

Now
$$  \langle \chi_{\overrightarrow{w}_1} |  \chi_{\overrightarrow{w}_2}
\rangle = \tau(\chi_{\overrightarrow{w}_1}^*
\chi_{\overrightarrow{w}_2}) =\mu_X (\overrightarrow{w}_1 \cap
\overrightarrow{w}_2), $$
and the previous lemma applies. 
\end{proof}

\begin{lem} For all $n>0$, we have 
$$\dim A_n=\dim H_n = 2g(2g-1)^{n-1}.$$ 
\end{lem}

\begin{proof} The space $A_{n}$ is spanned by the linear basis 
$\chi_{\overrightarrow{w}_\rho}$ with $w$ a word of exact length $
n$, since as in the proof of Lemma \ref{basis}, all functions corresponding to shorter words are dependent on these functions. An easy count gives the result: we pick the first letter from the alphabet on $g$ letters or its inverses, and consecutive letters with the condition that they differ from the terminal letter of the word already constructed. 
\end{proof}
We now construct a complete orthonormal basis for $H$ inductively, by adding to a basis of $H_{n}$ suitable elements of $H_{n+1}$ in the style of a Gram-Schmidt process. Initially, we set
$| \Psi_e \rangle = \chi_\Lambda$ and \begin{equation} \label{length1}  |\Psi_w \rangle := \frac{1}{\sqrt{\mu_X(\overrightarrow{w})}}\,
\chi_{\overrightarrow{w}} \ \ \ \ (|w|=1) \end{equation}
for $w$ running through a set $S$ of words of length one (viz., letters in the alphabet, and their inverses) unequal to one (arbitrarily chosen) letter. Set $I_1:=S \cup \{e\}$; then $\{|\Psi_w\rangle\}_{w \in I_1}$  is an orthonormal basis for $H_1$ by Lemma \ref{basis}. 

Now suppose $$\{|\Psi_w\rangle \, : \, w \in I_n \}$$ is our inductively constructed basis for $H_n$, where  $I_n$ is an index set. For every word $w$ of length $n$, choose a set $V_w$ of $2g-2$ letters
from the alphabet and its inverses, that are unequal to $t(w)^{-1}$, the inverse of the terminal letter of the fixed $w$, i.e., leave out one arbitrarily chosen letter from the possible extensions of $w$ to an admissible word of length $n+1$. Let $$I_{n+1} = I_n \cup \bigcup\limits_{|w|=n} V_w.$$ Figure \ref{figure3} has an example in the length $\leq 3$ words in the Cayley graph $Y_2$ for $g=2$.

\begin{figure}[h] \begin{center} \begin{picture}(120,120)\includegraphics{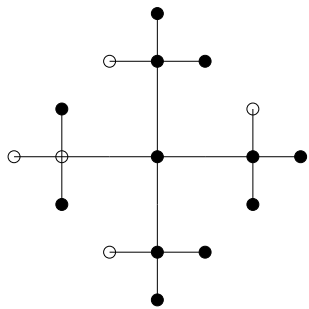}\end{picture} \end{center}
\begin{center} \caption{{ Black dots form a possible $I_3 \subseteq Y_2$}}\end{center}
\label{figure3}
\end{figure}

We claim that $\{\chi_{\overrightarrow{w}}\}_{w \in I_{n+1}-I_n}$ is a basis for $H_{n+1} \ominus H_n$. The functions are linearly independent since their supports are disjoint. Hence it suffices to check dimensions. But $$\dim (H_{n+1} \ominus H_n) = 2g(2g-1)^{n-1}(2g-2) = |I_{n+1}-I_n|.$$

We define $\{|\Psi_w \rangle\}_{w \in I_{n+1}}$ as the Gram-Schmidt orthonormalisation of $$\{ |\Psi_w\rangle \}_{w \in I_n} \, \cup \, \{ \chi_{\overrightarrow{w}} \}_{w \in I_{n+1} - I_n}.$$ 
We recall that this means we choose an enumeration  of the words in $I_{n+1}-I_n:$ $$ I_{n+1}-I_n = \{w_1,\dots,w_r\}$$ and set inductively for $i=1,\dots,r$
\begin{equation} \label{defPsi}  |\Psi_{w_i}\rangle = \frac{|\phi_{w_i} \rangle}{\| \phi_{w_i} \|} \end{equation}
with
\begin{equation} \label{defphi}  |\phi_{w_i} \rangle :=\chi_{\overrightarrow{w_{i}}} -\sum_{w \in I_n \cup\{w_1,\dots,w_{i-1}\}} 
\langle \Psi_w |\chi_{\overrightarrow{w_{i}}}\rangle |\Psi_w\rangle. \end{equation}
Then indeed,
$ \langle \Psi_v |\Psi_w\rangle =\delta_{v,w}, $
for all $|v|,|w|\leq n+1$.

Set $$I_\infty = \bigcup_{n \geq 0} I_n.$$ We use the complete basis $\{ |\Psi_w\rangle \}_{w \in I_\infty}$ of $H$ to compute the trace of a trace-class
operator $T$ in the form
$ \tr(T)=\sum \langle \Psi_w | T \Psi_w\rangle. $
With $T=aD^s$, we have
$$ \tr(aD^s)=1+\sum_w \langle \Psi_w |a \sum_{n \geq 1} \lambda_n^s
(P_n-P_{n-1}) \Psi_w \rangle . $$
Now the projector $P_n-P_{n-1}$ onto $H_n \ominus H_{n-1}$ is $\sum\limits_{r \in I_n - I_{n-1}}| \Psi_r \rangle \langle \Psi_r |$, so we get
$$ (P_n - P_{n-1}) |\Psi_w\rangle =\sum_{r \in I_n-I_{n-1}}| \Psi_r \rangle \langle \Psi_r |
\Psi_w \rangle = \delta_{|w|,n} \delta_{r,w} | \Psi_r \rangle =  \delta_{|w|,n} | \Psi_w \rangle.$$
Thus, we rewrite the above as
$$ \tr(aD^s)=1+\sum_{n \geq 1} \sum_{w \in I_n-I_{n-1}} \lambda_n^s  \langle \Psi_w |
a \Psi_w \rangle . $$
If we denote by 
$$ c_n(a) =\sum_{w \in I_n-I_{n-1}} \langle \Psi_w |
a  \Psi_w \rangle, $$
we can write
$$ \zeta_{a,X}(s) = \tr(a D^s)=1+ \sum\limits_{n\geq 1} \lambda_{n}^s\,
 c_{n}(a), \ \ \mathrm{Re}(s) \ll 0. $$ 

\begin{lem} \label{genusequal} If $ \zeta_{a,1}(s) = \zeta_{a,2}(s) $ for
$a=1=\chi_\Lambda$ the identity of $A_1$, respectively $A_2$, then
$g_1=g_2$. \end{lem} 

\begin{proof} We know that $\lambda_{n}=(\dim
A_{n})^3=(2g)^3(2g-1)^{3n-3}. $ By orthnormality, we find that $$c_n(1)= \sum_{|w| \in I_n-I_{n-1}} \langle \Psi_w | \Psi_w \rangle = \sum_{|w| \in I_n-I_{n-1}} 1 \, = \, 2g(2g-1)^{n-2}(2g-2).$$
Hence we find for $a=1$ that $$ \zeta_{1}(s) = 1 + \sum_{n \geq 1} \lambda_n^s c_n(1) = 1+ (2g)^{3s+1} \, \frac{2g-2}{2g-1} \, \sum_{n \geq 1} (2g-1)^{(3s+1)(n-1)},  $$  and thus \begin{equation} \label{zeta1} \zeta_1(s) = 1 +\frac{2g-2}{2g-1} \cdot \frac{(2g)^{3s+1}}{1-(2g-1)^{3s+1}}. \end{equation} 
For $a=1$, the 
condition $ \zeta_{a,1}(s) = \zeta_{a,2}(s) $ is thus equivalent to 
$$\frac{2g_1-2}{2g_1-1} \cdot \frac{2g_2-1}{2g_2-2} \cdot  \left( \frac{g_1}{g_2} \right)^{3s+1}  = \frac{1-(2g_1-1)^{3s+1}}{1-(2g_2-1)^{3s+1}} \mbox{ for } \mathrm{Re}(s) \ll 0 $$
If we let $s$ tend to $- \infty$, the right hand side tends to $1$. However, unless $g_1=g_2$, the left hand side tends to zero. This finishes the proof that $g_1=g_2$. 
\end{proof}

As mentioned in the remark above, we conclude from this lemma that the
algebras $A_1$ and $A_2$ are isomorphic via the induced Floyd
homeomorphism $\Phi: \Lambda_1 \rightarrow \Lambda_2$. 
 This makes the condition $ \zeta_{a,1}(s) = \zeta_{a,2}(s) $ meaningful.
 
\begin{lem}\label{equal} $c_{n,{1}}(a)=c_{n,{2}}(a)$ for all $a \in A_\infty$. 
\end{lem}  
\begin{proof}
The equality $\zeta_{a,1}(s) = \zeta_{a,2}(s) $ is equivalent to 
 $$\sum_{n \geq 0} \left( c_{n,1}(a)-c_{n,2}(a)\right) \lambda_n^{s} \equiv 0$$ 
for $\mathrm{Re}(s) \ll 0$. Here, $\lambda_n$ is the same for the two Riemann surfaces, since it only depends on their genus and those have just been shown to be equal in Lemma \ref{genusequal}. Now since all $\lambda_n$ are \emph{distinct}
positive integers, we also have an identically zero Dirichlet series 
$$ \sum_{N \geq 0} \til{c}_N N^s \equiv 0 \mbox{ for $\mathrm{Re}(s) \ll 0$} $$
with $\til{c}_N = c_{n,1}(a)-c_{n,2}(a)$ if $N=\lambda_n$ for some $n$, and
$\til{c}_N=0$ otherwise. Now clearly $\til{c}_N=0$ for all $N$, by the identity theorem for Dirichlet series (e.g., \cite{HW}, 17.1). 
\end{proof} 

\begin{lem} \label{algfunc}
For $a = \chi_{\overrightarrow{\eta}}$  and $w$ a word of length $n<|\eta|$, we have that 
$$ \langle \Psi_w | a \Psi_w \rangle = \mu(\overrightarrow{\eta}) \cdot \kappa$$
where $\kappa$ depends only on measures $\mu(\overrightarrow{v})$ of certain words $v$ of length $|v|<|\eta|$.
\end{lem}

\begin{proof} This holds for $w$ a word of length one, since by definition (\ref{length1}) and Lemma \ref{basis}, we have
$$  \langle \Psi_w | a \Psi_w \rangle = \frac{\mu(\overrightarrow{w} \cap \overrightarrow{\eta})}{\mu(\overrightarrow{w})}  =\left\{ \begin{array}{ll}  \frac{\mu(\overrightarrow{\eta})}{\mu(\overrightarrow{w})} & \mbox{if } w \subset \eta; \\ 0 & \mbox{otherwise.} \end{array} \right. $$
We then use induction on the word length of $w$. By construction of $\Psi_w$ (looking at the definitions in {Formul\ae}  (\ref{defphi}) and (\ref{defPsi})) it suffices to prove that 
for $w,u$ of length $ \leq n$, we have that $\langle\chi_{\overrightarrow{w}}|a\chi_{\overrightarrow{u}}\rangle$ is of the required form $\mu(\overrightarrow{\eta}) \cdot \kappa$
where $\kappa$ depends only on measures $\mu(\overrightarrow{v})$ of certain words $v$ of length $|v|<|\eta|$: $\Psi_w$ is a linear combination of such terms. Now
$$\langle\chi_{\overrightarrow{w}}|a\chi_{\overrightarrow{u}}\rangle = \mu(\overrightarrow{w} \cap \overrightarrow{\eta} \cap \overrightarrow{u}) =\left\{ \begin{array}{ll}  \mu(\overrightarrow{\eta}) & \mbox{if } w,u \subset \eta \\ 0 & \mbox{otherwise,} \end{array} \right. $$
since $\eta$ is longer than $w$ and $u$. This proves the claim.
\end{proof} 

Computing $c_{m-1}(\chi_{\overrightarrow{\eta}})$ as a linear combination of terms of the form $\langle \Psi_w | a \Psi_w \rangle,$ we find from Lemma \ref{algfunc} (now indicating the representation $\rho$, since we will soon vary it) :
\begin{equation} \label{inductmu} c_{m-1}(\chi_{\overrightarrow{\eta}_\rho}) =  \mu(\overrightarrow{\eta}_{\rho})  \cdot \kappa, \end{equation}
where $\kappa$ only depends on $\mu(\overrightarrow{v}_\rho)$ for $|v|<m$. 

We need one more technical observation, namely that $\kappa \neq 0$ in (\ref{inductmu}) or, what is the same: 
\begin{lem} $c_{m-1}(a) \neq 0$ for $a=\chi_{\overrightarrow{\eta}}$ with $|\eta|=m$. \end{lem}
\begin{proof}
Recall $c_{m-1}(a) = \sum\limits_{w \in I_{m-1}-I_{m-2}} \langle \Psi_w | a \Psi_w \rangle.$ The terms $$ \langle \Psi_w | a \Psi_w \rangle = \int |\Psi_w|^2 \cdot \chi_{\overrightarrow{\eta}} \, \, d\mu \geq 0$$ are all positive, but some might be zero. It therefore suffices to prove that at least one of them is non-zero, and for this, it suffices to find $w$ such that the support of $\Psi_w$ intersects $\overrightarrow{\eta}$. But if $x \in \overrightarrow{\eta}$, then there is a word $v$ of length $m-1$ such that $x \in \overrightarrow{v}$, too, since $$\bigcup\limits_{|v|=m-1} \overrightarrow{v} = \Lambda.$$ Then $\chi_{\overrightarrow{v}}(x)  \neq 0$, but,  as $\{\Psi_w\}_{w \in I_{m-1}-I_{m-2}}$ is a basis for $H_{m-1} \ominus H_{m-2}$, we have $$\chi_{\overrightarrow{v}}(x) = \sum_{w \in I_{m-1}-I_{m-2}} a_w \Psi_w(x)$$ for some coefficients $a_w$, hence there exists $w$ of length $m-1$ such that $\Psi_w(x) \neq 0$. 
\end{proof}

\begin{prop}
For all $\eta \in F_g$, $\mu_1(\overrightarrow{\eta}_{\rho_1}) =
\mu_2(\overrightarrow{\eta}_{\rho_2})$.  
\end{prop}

\begin{proof} We prove this by induction on the word length of $\eta \in F_g$
If $|\eta|=0$, we find $\overrightarrow{\eta}_{\rho_i}=\Lambda_i$ for $i=1,2$, so the identity holds. If it holds for all words of length $<m$, let $\eta$ denote a word of length $m$. 

Recall that the map $\Phi^*: A_1 \rightarrow A_2$ is such that
$\Phi^*(\chi_{\overrightarrow{w}_{\rho_1}} )=
\chi_{\overrightarrow{w}_{\rho_2}}$. 

We apply the expression (\ref{inductmu}) to both Riemann surfaces, substituting $\rho=\rho_1$ and $\rho=\rho_2$, respectively. Since $\kappa \neq 0$, this is a genuine formula for $\mu(\overrightarrow{\eta}_\rho)$. 
The equality $c_{m-1,1}(\chi_{\overrightarrow{\eta}_{{\rho_1}}})=c_{m-1,2}(\chi_{\overrightarrow{\eta}_{{\rho_2}}})$ is our assumption, and for the second factor on the right hand side we can inductively assume that the measures on the occuring words  of length $<m$ agree in both representations. Hence we find indeed $\mu(\overrightarrow{\eta_{\rho_1}})=\mu(\overrightarrow{\eta_{\rho_2}})$.
\end{proof}
This proposition finishes the proof of Theorem \ref{zetalambda}.
\end{proof}

\section{Rigidity from boundary isometry} 

We now prove Theorem \ref{tB} from the introduction: 

\begin{thm}
If $X_1$ and $X_2$ are compact Riemann surfaces of respective genus
$g_1, g_2 \geq 2$, such that  
$\zeta_{X_1,a}(s)=\zeta_{X_2,a}(s)$ for all $a \in A$, then $X_1$ and
$X_2$ are conformally or anti-conformally equivalent as Riemann
surfaces.\end{thm}  

\begin{proof} From Theorem \ref{zetalambda}, we find that $X_1$ and
$X_2$ have the same genus $g \geq 2$.  
We consider the two Schottky groups $\Gamma_1$ and $\Gamma_2$
corresponding to $X_1$ and $X_2$, respectively. Let $\rho_i :  F_g
\hookrightarrow \PGL(2,\C)$ denote the corresponding embeddings of the
abstract group $F_g$ (so $\rho_i(F_g)=\Gamma_i$), and let
$\alpha:=\rho_2 \circ \rho_1^{-1}$ denote the induced group
isomorphism $\Gamma_1 \rightarrow \Gamma_2$. We consider the map $\Phi
: \Lambda_1 \rightarrow \Lambda_2$ as in the diagram of the proof of
\ref{zetalambda}: $x\in \Lambda_1$ can be written as
$\iota_{\rho_1}(\lim w_i)$ for some Cauchy sequence $\lim w_i$ in
$Y_g$. We then define $\Phi(x):=\iota_{\rho_2}(\lim w_i)$. As was
remarked before, by Floyd's Lemma \ref{Floyd}, $\Phi$ is a
homeomorphism of $\Lambda_1$ onto $\Lambda_2$.  

\begin{lem} $\Phi$ is $\alpha$-equivariant. 
\end{lem}
\begin{proof} Given $\gamma \in \Gamma_1$, we can find $g \in F_g$
such that $\gamma=\rho_1(g)$. For $x = \iota_{\rho_1}(\lim w_i)$, we
find that $\gamma \cdot x = \iota_{\rho_1}(\lim gw_i)$. Hence
$$\Phi(\gamma \cdot x) = \iota_{\rho_2}(\lim gw_i) = \iota_{\rho_2}(g)
\cdot \iota_{\rho_2}(\lim w_i) = \rho_2(\rho_1^{-1}(\gamma)) \cdot
\Phi(x) = \alpha(\gamma) \Phi(x).$$ So we do find that $\Phi$ is
$\alpha$-equivariant. \end{proof}  

This means that $\Phi$ is a boundary homeomorphism in the sense of
Fenchel-Nielsen, see Tukia \cite{Tukia2}, 3C.  

By  Theorem \ref{zetalambda}, the equality of zeta functions implies
that $\Phi$ is an isometry. Indeed,  
$$\mu_2(\Phi^*(\chi_{\overrightarrow{w}_{\rho_1}})) =
\mu_2(\chi_{\Phi(\overrightarrow{w}_{\rho_1})})= 
\mu_2(\chi_{\overrightarrow{w}_{\rho_2}})=
\mu_1(\chi_{\overrightarrow{w}_{\rho_1}}),$$
where we use the definition of $\Phi$ in the second equality and the
proposition in the third. Thus, since
$\{\overrightarrow{w}_{\rho_i}\}$ is a basis for $\Lambda_i$, we find
that $\mu_2 \circ \Phi^* = \mu_1$.  

Now recall the following ergodic rigidity theorem: 

\begin{lem}[Chengbo Yue] Let $\Gamma_1$ and $\Gamma_2$ be
geometrically  finite subgroups in two simple, 
connected and adjoint Lie groups $G_1$ and $G_2$ of real rank one,
such that $\Gamma_1$ 
is Zariski dense in $G_1$. Let  $\alpha : \Gamma_1 \rightarrow
\Gamma_2$ be a type-preserving isomorphism. 
Then there exists a homeomorphism  $\phi: \Lambda_{\Gamma_1}
\rightarrow \Lambda_{\Gamma_2}$ which is equivariant 
with respect to  $\alpha$. If $\phi$ is absolutely continuous with
respect to the Patterson--Sullivan measure, then  $\alpha$ can be
extended to a continuous homomorphism $G_1 \rightarrow G_2$.\end{lem}

\begin{proof} This is literally Corollary B from  \cite{Yue}, apart
from the fact that the extended homomorphism $G_1 \rightarrow G_2$ can
be assumed \emph{continuous}, but this follows by looking at the
statement of Theorem A from which the corollary follows. \end{proof} 

We want to apply this corollary with $G_1=G_2=\PGL(2,\C)$ and
$\Gamma_i$ our Schottky groups, so let us check the conditions:  
Both Schottky groups are geometrically finite subgroups of $\PGL(2,\C)$;
$\PGL(2,\C)=\PSL(2,\C)$ is a simple and connected adjoint
real-rank-one Lie group; 
and finally:

\begin{lem} A noncommutative Schottky group is Zariski dense in $\PGL(2,\C)$. 
\end{lem}

\begin{proof} Since the group operations on such a Schottky group are
induced from the algebraic operations on the algebraic group
$\PGL(2)$, the Zariski closure $\hat{\Gamma}$ is itself an algebraic
subgroup of $\PGL(2)$. Assume $\hat{\Gamma}$ is a strict algebraic
subgroup of $\PGL(2)$. Let $\hat{\Gamma}_0$ denote its connected
component of the identity. The group of connected components
$\hat{\Gamma}/\hat{\Gamma}_0$ is finite, and $\Gamma \cap
\hat{\Gamma}_0$ 
is a finite index subgroup of the free group $\Gamma$, hence free of
the same rank $g$; and its Zariski closure is connected. It suffices
that this group has full Zariski closure, hence we can assume without
loss of generality that $\hat{\Gamma}$ is connected. However, if
$\hat{\Gamma}$ is connected of dimension $\leq 2$, then it is solvable
(cf.\ \cite{LAG}, IV.11.6), and a solvable group cannot contain a free
group of rank $g \geq 2$ (since the composition series of
$\hat{\Gamma}$ would descend to one for $\Gamma$). On the other hand,
if $\dim \hat{\Gamma} = 3$, then since $\PGL(2)$ is connected, we have
$\hat{\Gamma}=\PGL(2)$. \end{proof} 

Since $F_g$ has no parabolic points, we know that the equivariant
boundary homeomorphism $\Phi$ is unique and type-preserving (Tukia,
\cite{Tukia1}, p.\ 426), hence it coincides with the boundary
homeomorphism $\phi$ in Yue's result. Since all conditions are
satisfied, we can apply the result (replacing $\phi$ by $\Phi$) to
both the isometry $\Lambda_1 \rightarrow \Lambda_2$ and its inverse,
we find that $\alpha$ extends to a continuous group automorphism
$\PGL(2,\C) \rightarrow \PGL(2,\C)$. Now recall that the automorphisms
of $\PGL(2,k)$ over a field $k$ have been classified by Schreier and
van der Waerden (cf.\ \cite{SW}, see also the supplement to
\cite{Hua}): the outer automorphisms are induced from field
automorphisms of $k$. Now all  continuous field automorphisms of $\C$
fix $\R$. 

We conclude that there is an isomorphism $\Gamma_1 \rightarrow
\Gamma_2$ of the form $$\gamma_1 \rightarrow g \gamma_1^\sigma
g^{-1}$$ for $g \in \PGL(2,\C)$ and $\sigma \in \Aut(\C/\R)$, that is,
$\Gamma_1$ and $\Gamma_2^\sigma$ are conjugate in $\PGL(2,\C)$. Now
note that $\Gamma_2^\sigma$ uniformizes the curve $X_2^\sigma$: 
$$ (\P^1(\C)-\Lambda_{\Gamma_2^\sigma})/ \Gamma_2^\sigma = \left(
(\P^1(\C)-\Lambda_{\Gamma_2})/ \Gamma_2\right)^\sigma = X_2^\sigma.$$ 
 Hence $X_1$ and $X_2^\sigma$ are isomorphic Riemann surfaces, so
$X_1$ and $X_2$ are conformally or anti-conformally equivalent, the
former case arising when $\sigma$ is trivial, and the latter case
arising when $\sigma$ is complex conjugation. 
\end{proof}

\begin{rem}
The statement that two hyperbolic compact Riemann surfaces of the same
genus are isomorphic if and only if the boundary map $S^1 \rightarrow
S^1$ induced from the isomorphism of their fundamental groups, seen as
Fuchsian groups of the first kind, is absolutely continuous in
Lebesgue measure is originally part of Mostous rigidity theorem (cf.\
Mostow \cite{Mostow} 22.14, p.\ 178). An easy proof for this case is
in Kuusalo \cite{Kuusalo}, and  Bowen has given another proof using
Gibbs measures in \cite{Bowen}. For more general M\"obius groups
(whose limit set is not necessarily the full boundary of the symmetric
space, such as Schottky groups, and in higher dimensions), there is
the work of Sullivan (\cite{Sullivan}) and Tukia (\cite{Tukia1},
\cite{Tukia2}). The typical ergodic rigidity theorem in this setting
is that an absolutely continuous boundary map is identical to the
restriction of a M\"obius transformation \emph{on the limit set}. What
happens outside the limit set, however, depends on other
considerations (cf. \cite{Tukia2}, Marden \cite{Marden}). See, e.g.,
Tukia (\cite{Tukia1}), Section 4D for a 3-dimensional example of an
isomorphism of Schottky groups with absolutely continuous boundary
map, that does \emph{not} extend to a M\"obius map outside the limit
set.\end{rem} 

\begin{rem}
The construction cannot be extended to the crossed product boundary
operator algebra $C(\Lambda) \rtimes \Gamma$, since Connes
hyperfiniteness result \cite{Conneshyper} and the hyperbolic growth of
the group $\Gamma$ ($g \geq 2$) prevent this algebra from carrying
finitely summable spectral triples.  
However, $C(\Lambda) \rtimes \Gamma$ can 
be identified with a Cuntz--Krieger algebra (\cite{CK}), which has a
standard AF-subalgebra $\til{A}$. This has a maximal abelian
subalgebra that can be identified with our algebra $A$ (\cite{CK},
2.5). One can construct a conditional expectation $E$ from $\til{A}$
to $A$. Thus,  the construction of the spectral triple extends to this
AF-algebra $\til{A}$ by using an \emph{unfaithful} state $\tau \circ
E$ that is zero outside $A$ and equals our state $\tau$ on $A$. Since
this construction, however, is just a ``factorisation'' through the
above commutative spectral triple, it doesn't give a lot of new
information. 
\end{rem}

\section{Remarks and Questions} 

\begin{se} 
 An interesting question (e.g., in the light of Arakelov geometry) is
how to generalize the result to the case of $p$-adic Mumford curves
(\cite{GvdP}, \cite{Mumford}). See \cite{CMRV} for some results on
trees, inspired by earlier work on non-finitely summable spectral
triples for tree actions by Consani and Marcolli
\cite{ConsaniMarcolli}. 
\end{se}
\begin{se}  Is there such a theory for Fuchsian uniformisation instead
of Schottky uniformisation? In the noncompact case? 
\end{se}
\begin{se} There are other ways of looking at Riemann surfaces from
the point of view of non-commutative geometry. For example, as a
commutative conformal manifold, it carries the non-commutative
differential geometry of ``quantized calculus'', whose Fredholm module
determines the conformal isomorphism type of the surface (Connes,
Donaldson, Sullivan, N.~Teleman \cite{Connesbook},
IV.4.$\alpha$). Also, Consani and Marcolli (\cite{CMinf}) have
constructed $\theta$-summable, but non finitely summable spectral
triples from the boundary action, where the Hilbert space is a
symmetrized version of the $L^2$-space of the boundary that should
have Hausdorff dimension $<1$. There is the work of B\"ar
(\cite{Baer}) mentioned before. All of these constructions are rather
different from the one in this paper. In particular, our spectral
triple is finitely summable, so better suited to tools such as the
local index formula of Connes and Moscovici (\cite{CoMo}).  
Nevertheless, the question arises whether the Consani-Marcolli
spectral triple hears the conformal shape of the Riemann surface.  B\"ar's spectral tiple does enjoy this property.
\end{se}
\begin{se}  Does our spectral triple $\cS_X$ carry a real structure?
Does it then arise from a commutative spin manifold (\cite{GVF},
\cite{RV})? Can one enhance $\cS_X$ by additional classical structure,
so this structure determines the conformal structure of $X$ completely
(not just conformal or anti-conformal)? 
Notice that in our construction we work only with the absolute value
of the Dirac operator, and we use zeta functions that do not see the
sign. In order to relate it to spectral triples coming from a
Riemannian manifold one would need to first enrich it with a sign
(which gives the fundamental class in $K$-homology) and then with a
compatible real structure. It seems that the information on the sign
will be needed to reconstruct completely the conformal structure. 
\end{se}

\begin{se}  Since special values of the spectral zeta functions are
just 0-Hochschild homology, it is interesting to compute higher
characteristic classes of the spectral triple and relate them to the
actual geometry of the original Riemann surface. This is especially
tempting in arithmetically interesting cases, such as modular curves. 
\end{se}
\begin{se}  Can Theorem \ref{tB} be extended to an injective functor
from the category of Riemann surfaces with conformal and
anti-conformal morphisms to a category of spectral triples? 
\end{se}

%%%%%%%%%%%%%%%%%%%%%%%%%%%%%%%%%%%%%%%%%%%%%%%%%%%%%%%%%%%%
% Appendices, if necessary...
%\appendix
%%%%%%%%%%%%%%%%%%%%%%%%%%%%%%%%%%%%%%%%%%%%%%%%%%%%%%%%%%%%

%%%%%%%%%%%%%%%%%%%%%%%%%%%%%%%%%%%%%%%%%%%%%%%%%%%%%%%%%%%%
% bibliography
%%%%%%%%%%%%%%%%%%%%%%%%%%%%%%%%%%%%%%%%%%%%%%%%%%%%%%%%%%%%

\begin{small}
\bibliographystyle{plain}
\bibliography{zetas}
\end{small}

%%%%%%%%%%%%%%%%%%%%%%%%%%%%%%%%%%%%%%%%%%%%%%%%%%%%%%%%%%%%
% prints affiliation at the end 
%%%%%%%%%%%%%%%%%%%%%%%%%%%%%%%%%%%%%%%%%%%%%%%%%%%%%%%%%%%%

\finalinfo 

\end{document}